\RequirePackage[reqno]{amsmath}
\documentclass[a4paper,12pt]{amsart}

\usepackage{amssymb,mathtools}
\usepackage[reqno,fleqn,intlimits]{empheq}

\usepackage{cite,color}
\usepackage{graphicx,xcolor}
\usepackage{float}

\newcounter{mnotecount}[section]

\newcommand{\rmnote}[1]{}

\DeclareFontFamily{OT1}{rsfs}{}
\DeclareFontShape{OT1}{rsfs}{m}{n}{ <-7> rsfs5 <7-10> rsfs7 <10-> rsfs10}{}
\DeclareMathAlphabet{\mathscr}{OT1}{rsfs}{m}{n}

\newcommand{\bel}[1]{\begin{equation}\label{#1}}
\newcommand{\beal}[1]{\begin{eqnarray}\label{#1}}
\newcommand{\beadl}[1]{\begin{deqarr}\label{#1}}
\newcommand{\eeadl}[1]{\arrlabel{#1}\end{deqarr}}
\newcommand{\eeal}[1]{\label{#1}\end{eqnarray}}
\newcommand{\eead}[1]{\end{deqarr}}
\newcommand{\eea}{\end{eqnarray}}
\newcommand{\eeaa}{\end{eqnarray*}}

\newcommand{\be}{\begin{equation}}
\newcommand{\ee}{\end{equation}}

\DeclareFontFamily{OT1}{rsfs}{}
\DeclareFontShape{OT1}{rsfs}{m}{n}{ <-7> rsfs5 <7-10> rsfs7 <10->
rsfs10}{} \DeclareMathAlphabet{\mycal}{OT1}{rsfs}{m}{n}

\newcommand{\N}{{\mathbb N}}

\newcommand{\Ric}{\operatorname{Ric}}

\def\mysavedown#1{\edef\mysubs{\mysubs#1}}
\def\mysaveup#1{\edef\mysups{\mysups#1}}
\def\mydown#1{{\mytensor}_{\vphantom{\mysubs}#1}}
\def\myup#1{{\mytensor}^{\vphantom{\mysups}#1}}
\def\tensor#1#2{
  #1
  \def\mytensor{\vphantom{#1}}
  \def\mysubs{\relax}
  \def\mysups{\relax}
  \let\down=\mysavedown
  \let\up=\mysaveup
  #2
  \let\down=\mydown
  \let\up=\myup
  #2
  }

\newcommand{\Riem}{\operatorname{Riem}}

\newcommand{\Tr}{\operatorname{Tr}}

\newcommand{\R}{\mathbb R}

\renewcommand{\S}{\mathbb S}

\newcommand{\griem}{{ {\mathfrak g}}}

\renewcommand{\phi}{\varphi}
\renewcommand{\epsilon}{\varepsilon}

\renewcommand{\Im}{\mbox{Im}\;}

\def\crn#1#2{{\vcenter{\vbox{
        \hbox{\kern#2pt \vrule width.#2pt height#1pt
           }
          \hrule height.#2pt}}}}

\newcommand{\Ein}{\operatorname{Ein}}

\renewcommand{\hbar}{{\overline h}}

\newcommand{\pre}[2]{{{\vphantom{#2}}^{#1}}\kern-.2ex{#2}}
\newcommand{\gpsea}{\mathfrak G}
\newcommand{\gpseb}{\mathcal G}

\theoremstyle{plain}
\newtheorem{theorem}{Theorem}[section]
\newtheorem{lemma}[theorem]{Lemma}
\newtheorem{proposition}[theorem]{Proposition}

\theoremstyle{definition}
\newtheorem{exemple}[theorem]{Exemple}

\numberwithin{equation}{section}

\begin{document}
\title[$L^p$ almost conformal isometries and Ricci operator]
{$L^p$ almost conformal isometries of Sub-Semi-Riemannian metrics and Solvability of  a Ricci equation}

\author[E.  Delay]{Erwann
Delay} \address{Erwann Delay, Avignon Universit\'e, Laboratoire de Math\'ematiques d'Avignon (EA 2151)
F-84916 Avignon}
\email{Erwann.Delay@univ-avignon.fr}
\urladdr{http://www.math.univ-avignon.fr}

\date{January 18, 2107}

\begin{abstract}
Let $M$ be a smooth compact  manifold without boundary. 
We consider  two smooth Sub-Semi-Riemannian metrics  on $M$.
Under suitable conditions, we show that they are almost conformally isometric
in an $L^p$ sense. 
Assume also that $M$ carries a Riemannian metric with  parallel Ricci curvature. 
Then 
 an equation of Ricci type, is in some sense solvable, without assuming any
closeness near a special metric.
\end{abstract}


\maketitle

\\
\newline

\tableofcontents

\section{Introduction}\label{section:intro}
The goal of this note is to prove that the two principal results of D.~DeTurck \cite{DeturckEinstein} given for positive definite symmetric bilinear form and for 
special Einstein metrics  can be extended significantly in
different ways.
%

Firstly, we can extend  the positive definideness  condition of the Riemannian metrics to  Sub-Semi-Riemmannian metrics with the same rank and signature.


Secondly we are able to replace some particular Einstein metrics of non zero  scalar curvature by any parallel Ricci metrics (ie. metrics with covariantly constant Ricci tensor).\\

Let $M$ be a smooth compact  manifold without boundary. 
A Sub-Semi-Riemannian metric $\gpsea$ (SSR-metric for short)  is a symmetric covariant 2-tensor field with constant
signature and constant rank.


Let us now state the first result, interesting by itself, about almost conformal  isometries.

\begin{lemma}\label{maintheoremPseudoRiem} 
Assume that  $\gpsea$ and $\gpseb$ are two smooth SSR-metrics on $M$  
with the same rank and signature. 
 Let  $g$ be a smooth Riemmannian metrics on $M$, let $p\in[1,\infty)$ and let $\epsilon>0$.
Then   there exist a smooth diffemorphism $\Phi$  and a smooth positive function $f$ such that $\Phi^*(f\gpseb)-\gpsea$
is $\epsilon$-close to zero in the $L^p$ norm relative to $g$.
\end{lemma}

Before going to the application for a Ricci equation, let us introduce some notations.
For $(M,\griem)$ a smooth riemannian  manifold, we denote by  $\Ric(\griem)$ its Ricci  curvature.
For a real constant $\Lambda$, we consider the operator
$$
\Ein(\griem):=\Ric(\griem)+\Lambda \griem.
$$
This operator is geometric in the sense that for any smooth diffeomorphism $\varphi$, 
$$
\varphi^*\Ein(\griem)=\Ein(\varphi^*\griem).
$$
We would like to invert $\Ein$.
We thus choose $\mathcal E$ a symmetric 2-tensor field on $M$,  and look for $\griem$ Riemannian metric such that
\bel{mainequation}
\Ein(\griem)=\mathcal E.
\ee
This is a geometrically natural and difficult quasilinear system to solve, already for  perturbation methods.
The prescribed Ricci curvature problem has a long history  starting with the work of D. DeTurck \cite{{Deturck:ricci}},\cite{DeturckEinstein}, \cite{Deturckrank1},
\cite{Deturck-Koiso}, \cite{Deturck-Kazdan}, \cite{Baldes1986},
 \cite{Hamilton1984}, \cite{Delanoe1991}, \cite{Delanoe2003}, \cite{Delay:etude},\cite{DelayHerzlich},   \cite{Delay:study}, \cite{Delay:ricciproduit}, \cite{Delay:ricciAE},...

Motivated by the explosion of studies around the Ricci flow, and recently, some discrete versions thereof  (eg. $\Ein(g_{i+1})=g_i$), a  renewed interest arises for this kind of natural geometric equations. We invite the reader to look at the nice recent works of A. Pulemotov and Y. Rubinstein  \cite{Pulemotov2013} and \cite{PulemotovRubinstein}
for related results. Our contribution here is the following.

\begin{theorem}\label{maintheoremRicci} 
Assume that $M$ carries a Riemannian metric $g$ with parallel Ricci tensor.
Let $\Lambda\in \R$ such that  $\Ein(g)$ is non degenerate, and that  $-2\Lambda$ is not  in the spectrum of the Lichnerowicz Laplacian of $g$.
\footnote{Like D. Deturck \cite{DeturckEinstein}, we may allow an eigenspace spanned by $g$ when $\Lambda=0$}
Then for any  ${\mathcal E}\in C^{\infty}(M,\mathcal S_2)$  
with the same rank and signature as  $\Ein(g)$ 
at each point of $M$,
 there exist a smooth positive function $f$ and  a Riemannian metrics $\griem$ in 
$C^{\infty}(M,\mathcal S_2)$ such that
$$
\Ein(\griem)=f\,{\mathcal E}.
$$
\end{theorem}
The proof goes by combining   the Lemma \ref{maintheoremPseudoRiem} ,  the local inversion result of Proposition \ref{einsclose} for weak regular metric (where the conformal factor  $f$ is not required) and 
a regularity  argument.
We have then solved the problem up to a positive function $f$. Here  we do not expect  that $f$  can be taken equals to one in general,
this will be the subject of future investigations.

Parallel Ricci metrics,  are (locally) products of Einstein metrics (see eg. 
\cite{Wu:Holonomy}).
They exists on the simplest examples  of manifolds who do not admit Einstein metrics,
like $\S^1\times\S^2$, $\Sigma_g\times\S^2,$ ($g\geq1$) or $\Sigma_g\times\mathbb T^2,$ ($g\geq2$) where $\Sigma_g$ is a surface of genius
$g$. They are also static solutions of some geometric fourth order flows (eg. $\partial_tg=\Delta_g\Ric (g)$).
Finally they are particular cases of Riemannian manifolds with Harmonic curvature (or equivalently Codazzi Ricci tensor).

Our global  result show once again that  such  metrics with covariantly constant Ricci tensor deserve a particular attention.

\medskip

{\small\sc Acknowledgments} :  I am grateful to Philippe Delano\"e for comments
and improvements, and to Alexandra Barbieri and Fran\c{c}ois Gautero
for the picture of the simplex.

\section{$L^p$ closeness  of some Sub-Semi-Riemannian metrics}
We follow the section 3 called "approximation lemma" in \cite{DeturckEinstein}
in order to verify that all the step there can be adapted for SSR-metrics
as above. This will prove the Lemma \ref{maintheoremPseudoRiem}.

 We will keep almost the same notations as in \cite{DeturckEinstein}, just replacing  $S$ and $R$ there respectively by $\gpsea$ and $\gpseb$ here.\\

Let $\gpsea$ and $\gpseb$ be as in the introduction,  
we thus assume they have the same signature and the same rank. 
For the rest of the section we fix a Riemannian metric $g$, an $\epsilon>0$
and $p\in[1,+\infty)$.
All the measures, volumes,  and norms are understood with respect to $g$.\\

At each point $x\in M$, 
the two SSR-metric $\gpsea$ and $\gpseb$
having  the same rank and signature,  there exists  an orientation preserving 
automorphism $u_x$ of $T_xM$, 
such that
$$
\gpseb_x(u_x(.),u_x(.))=\gpsea_x(.,.).
$$
For $x\in M$, the following  construction can be performed  using the $g$-exponential map at $x$. 
 There exists an open set $U_x$ such that :\\

\noindent$(i)$  $U_x$ is contained in a  
coordinate neighborhood of $x$ where
in this coordinate (centered at 0), 
up to a positive automorphism $u_x$ of $T_xM$,
$\gpseb$ is equal to $\gpsea$ at $x$:
$$
^tu_x\,\gpseb_x\,u_x=\gpsea_x,
$$
(ii) For any positive real $\alpha_x$, the linear change of coordinates
$$
\Phi_x:=\sqrt{\alpha_x}\,u_x
$$
satisfies  on $U_x$ the estimate (the left hand side of which does not depend upon $\alpha_x$, and vanishes at the origin),
\bel{inegaii}
\left|(\Phi_x^*\frac1{\alpha_x}\gpseb)_y-\gpsea_y\right|^p\leq\min\left(\frac{\epsilon^p}{2\text{Vol}(M)},
|\gpsea_y|^p\right).
\ee
We consider a triangulation of $M$ where each simplex $S$ lies in the interior of some
$U_x$ with $x\in \mathring S$. Since the point $x$ belong to the  {\it interior} of the simplex $S$, shrinking $\alpha_x$ if necessary, we are sure that
$\Phi_x$ send $S$ into $S$ (the norm of $\Phi_x$ approaches  zero when $\alpha_x$ tends to zero).\\

Let $\Omega$, $\Omega_1$, $\Omega_2$, $\Omega_3$ be some open neighbourhoods of the  $(n-1)$ dimensional skeleton (composed with union of the boundary of all simplex) 
with the properties :
$$
\text{Vol}(\Omega)<\frac{\epsilon^p}{2(\max_M|\gpseb|+2\max_M|\gpsea|)^p},
$$
and 
$$
\Omega_3\subset\overline\Omega_3\subset
\Omega_2\subset\overline\Omega_2\subset\Omega_1\subset\overline\Omega_1\subset
\Omega.
$$
The rest of the proof in section 3 of  \cite{DeturckEinstein} is based on triangular inequalities 
between norms of tensors and can be implemented 
here without any change.
For a better understanding, though, we provide further details of the figure page 368 of
\cite{DeturckEinstein}, specifying the estimates that occur
 on the differents parts of the simplex,
see figure \ref{FigureSimplex}.
On  the picture, we have  denoted the error $|\Phi^*(f\gpseb)-\gpsea|$  by $e$ :
$$
e=|\Phi^*(f\gpseb)-\gpsea|.
$$
On  the inner part $T$ of the simplex,  $e$ is estimated by (\ref{inegaii}).
The transition of the diffeomorphism $\Phi$, on the middle ring $R_2=S\cap(\Omega_1\backslash \Omega_2)$,
from $\Phi_x$ to the identity, still exist  because our $\Phi_x=\sqrt{\alpha_x}\,u_x$ is  an orientation preserving map with norm less
than 1 as in \cite{DeturckEinstein}.
\begin{figure}[H]
   \includegraphics[ angle=-90,scale=0.6]{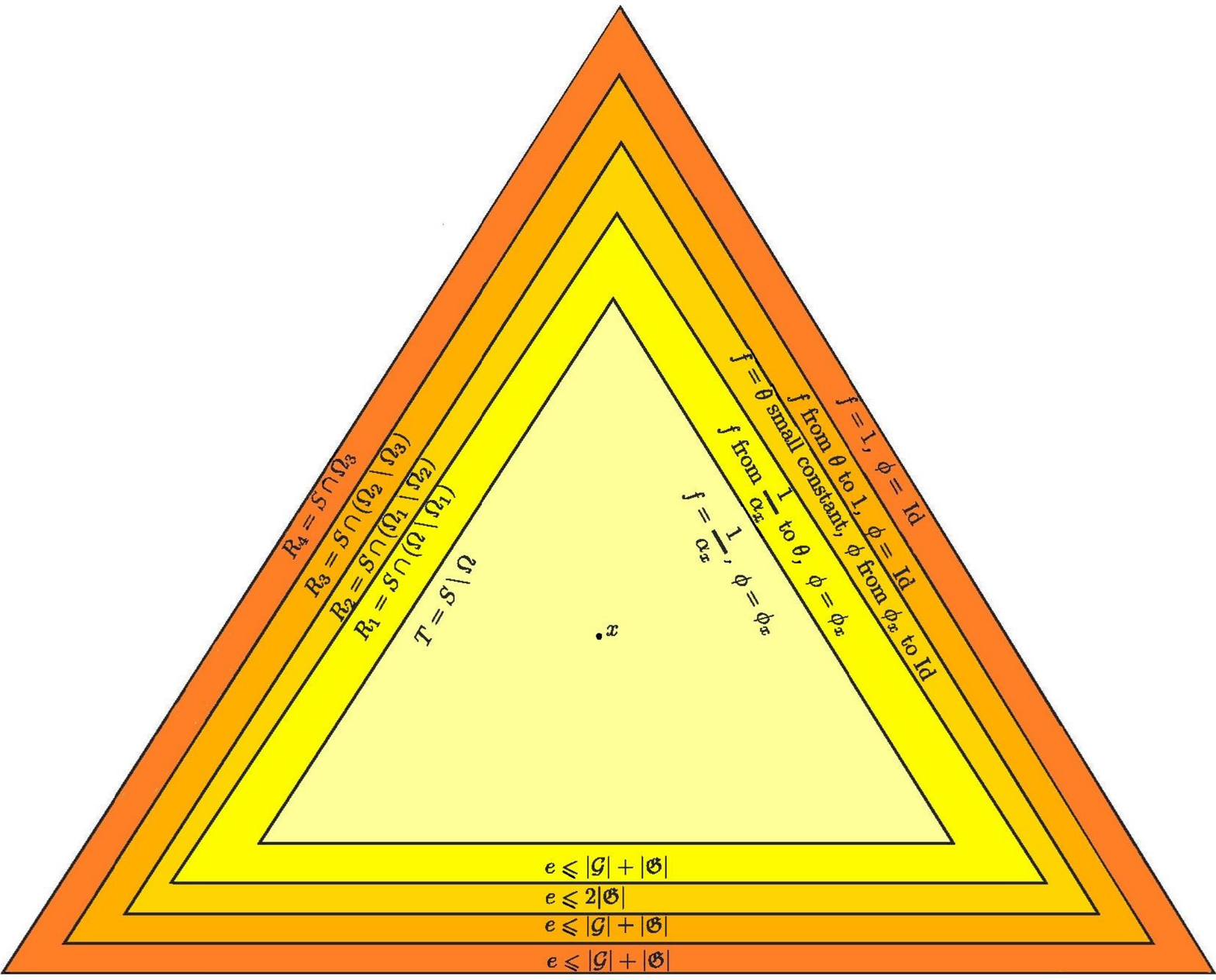}
	\vspace{-5cm}
	\caption{The simplex $S$ with the values of $f$ and $\Phi$, and the estimates  of $e$.}
	\label{FigureSimplex}
\end{figure}
\begin{exemple}
The simplest non trivial example consists  of a product of manifolds 
$M=\mathcal X\times \mathcal Y\times \mathcal Z$ with the two SSR-metrics of the form
$$
\gpsea(x,y,z)=-g_{\mathcal X,y,z}(x)\oplus g_{\mathcal Y,x,z}(y)\oplus 0_{\mathcal Z},
$$
where $g_{\mathcal X,y,z}$ is a family of Riemannian metrics on $\mathcal X$,
depending on the parameters $y,z$ and smooth in all of its arguments.
Here, any of the three manifolds  but one may  be  reduced to  a point.
\end{exemple}
\section{Solvability of a Ricci type equation}
We revisit the section 2 of \cite{DeturckEinstein}
called ``perturbation lemma''.\\

We first need to introduce some operators.
The divergence of a symmetric 2-tensor field and its $L^2$ adjoint acting on one form are
$$
(\delta h)_j:=-\nabla^i h_{ij}\;,\;\;\;\;(\delta^*v)_{ij}:=\frac12(\nabla_iv_j+\nabla_jv_i).
$$
The gravitationnal operator acting on symmetric 2-tensors is 
$$
G(h):=h-\frac12\Tr_g(h)g.
$$
The Lichnerowicz Laplacian is \footnote{Different sign convention with DeTurck}
$$
\Delta_L=\nabla^*\nabla+2\Ric-2\Riem.
$$
It appears  in  the Linearization of the Ricci operator :
$$
D\Ric(g)=\frac12\Delta_L+\delta^*\delta\,G.
$$
The Hodge Laplacian acting on one forms is
$$
\Delta_H=\Delta+\Ric=\nabla^*\nabla+\Ric=d^*d+dd^*.
$$
We also define the following Laplacian 
$$
\Delta_V:=2\,\delta\,G\,\delta^*=\nabla^*\nabla-\Ric=\Delta_H-2\Ric.
$$
We denote by $V$ its finite dimensionnal kernel, composed of smooth one forms (by elliptic 
regularity).\\

We start with the equivalent of proposition 2.1 in \cite{DeturckEinstein}.
\begin{proposition}\label{einsclose}
Let $(M,g)$ be a smooth Riemannian manifold with parallel Ricci curvature.
Let $\Lambda\in\R$ such that  $\Ein(g)$ is non degenerate and  that 
$-2\Lambda$ is not in the spectrum of the Lichnerowicz Laplacian. 
Let  $k\in\N$ and $p>n$.  
Then for any  ${\mathcal E}$  close to $\Ein(g)$ in $H^{k+1,p}(M,\mathcal S_2)$, there exist a Riemannian metrics $\griem$ in 
$H^{k+1,p}(M,\mathcal S_2)$ such that
$$
\Ein(\griem)={\mathcal E}.
$$
\end{proposition}
In \cite{DeturckEinstein}, the
 proof of the corresponding proposition is given by a succession of lemmas. We thus revisit them  one after the other.
Some care is needed because we have to replace Ric and $-\Delta_L$  there, respectively with Ein 
and  $\Delta_L+2\Lambda$ here.
Furthermore, in our context, the  operator $\Delta_L+2\Lambda$ has no kernel 
whereas the kernel of $\Delta_L$ is nonempty  in \cite{DeturckEinstein}, spanned by $g$.
We clearly have also  for any Riemannian metrics $g$, the Bianchi identity
$$\delta G(\Ein(g))=0.$$

We start with a local study of the action of the diffeomorphim group
on the covariant symetrics 2-tensors, near a non degenerate parallel one.
The result obtained remains  in the spirit   of  the local study near a Riemmanian  metric by Berger, Ebin, or Palais  (see eg. the  lemma 2.3 of \cite{DeturckEinstein}). Here the metric tensor is replaced by a non degenerate parallel tensor field.
\begin{lemma}\label{lemme23}
Let  $E$ be a smooth, non degenerate and parallel symmetric two tensor field.
Let $\mathcal X$ be a smooth Banach submanifold of $H^{k,p}(M,\mathcal S_2)$, whose tangent 
space at $E$ is a complementary of $\delta^*(H^{k,p}(M,\mathcal T_1))$.
Then for any ${\mathcal E}$ close enough to $E$ in $H^{k,p}(M,\mathcal S_2)$, there exist
an  $H^{k+1,p}$ diffeomorphism $\Phi$ close to the identity such that $\Phi^*{\mathcal E}\in\mathcal X$.
\end{lemma}
\begin{proof}
The tensor field $E$ being parallel, its Lie derivative  in the direction of a vector field $v$ is
$$
\mathcal L_{v}E=2\delta^*(Ev).
$$
Locally, the submanifold $\mathcal X$ can be seen as the image of an immersion
$\mathfrak X:U\longrightarrow H^{k,p}(M,\mathcal S_2)$, with $\mathfrak X(0)=E$.
We define $\mathcal T^\bot$ to be the set of vector fields 
$v\in H^{k+1,p}(M,\mathcal T_1)$ such that $E\,v$ is $L^2$-orthogonal \footnote{closed complementary suffice}
to $\ker \delta^*$.

Let 
$$
F:U\times\mathcal T^\bot\times H^{k,p}(M,\mathcal S_2)\longrightarrow
H^{k,p}(M,\mathcal S_2),
$$
defined by
$$
F(k,Y,{\mathcal E})=\Phi^*_{Y,1}({\mathcal E})-\mathfrak X(k),
$$
where $\Phi_{Y,1}$ is the flow of the vector field $Y$ at time 1.
We have $F(0,0,E)=0$ and the linearisation of $F$ in the first two variables is
$$
D_{(k,Y)}F(0,0,E)(l,X)=2\delta^*(E\,X)-D\mathfrak X(0)l.
$$
Now because 
$$
H^{k,p}(M,\mathcal S_2)=\delta^*(H^{k+1,p}(M,\mathcal T_1))\oplus \Im D\mathfrak X(0),
$$
and $E$ is non degenerate, then $D_{(k,Y)}F(0,0,E)$
is an isomorphism. From  the implicit function theorem, for ${\mathcal E}$
close to $E$, there exist $k$ and $Y$ small such
that $F(k,Y,{\mathcal E})=0$.
\end{proof}
Let us recall  the lemma 2.5 in \cite{DeturckEinstein} :\footnote{
It seems there is a misprint in the proof this lemma:
The Ricci term for $\delta G \delta^*$ at the top of page 362  in \cite{DeturckEinstein} has a different sign.}

\begin{lemma}\label{lemmegrossplit}
For $k\geq1$, 
we have 
$$
H^{k,p}(M,\mathcal S_2)=\frac{(\ker \delta G \cap H^{k,p}(M,\mathcal S_2))}{\delta^*(V)}\oplus \delta^*(H^{k+1,p}(M,\mathcal T_1))\oplus G\delta^*(V).
$$
\end{lemma}

The equivalent of the lemma 2.6 
 in \cite{DeturckEinstein} becomes (we do not have to quotient by $R\,g$ because 
of no kernel for us, and so we do not need to adjust with a constant $c$).
\begin{lemma}\label{imric}
Suppose $k\geq 0$, $p>n$, and ${g}$ satisfies the hypotheses of theorem \ref{maintheoremRicci} let
$$
K=\frac{(\ker \delta G \cap H^{k+2,p}(M,\mathcal S_2))}{\delta^*(V)}\oplus G\delta^*(V)
$$
and define $F:K\longrightarrow H^{k,p}(M,\mathcal S_2)$ by\footnote{To avoid ambiguities,
we may take any fixed closed complementary $W$  to $\delta^*(V)$ in $\ker \delta G \cap H^{k+2,p}(M,\mathcal S_2)$ instead of the first factor of $K$}
$$
F(b):=\Ein(g+b).
$$
Then for some neigbborood $U$ of $0$, $F(U)$ is a Banach submanifold of $H^{k,p}(M,\mathcal S_2)$ whose  tangent space at $F(0)=\Ein(g)$ is a complementary space of 
$\delta^*(H^{k+1,p}(M,\mathcal T_1))$.
\end{lemma}
\begin{proof}
We have to show that the derivative of $F$ at $g$ is injective and its image has the closed subspace $\delta^*(H^{k+1,p})$ as complementary.
A metric with parallel Ricci tensor  is a  local Einstein product so is smooth.
We first show that the spaces $\Im \delta^*$ and $\ker \delta G$ are ``stable'' (modulo two points of regularity)  by $\Delta_L+2\Lambda$ but also by $D\Ein(g)$ (when the metric is Ricci parallel).
Indeed, recall that  in that case \cite{Lichnerowicz:prop}:
$$
\delta\,\Delta_L=\Delta_H\,\delta,
$$
and the adjoint version :
$$
\Delta_L\,\delta^*=\delta^*\,\Delta_H.
$$
We deduce that  
$$
(\Delta_L+2\Lambda)\,\delta^*=\delta^*\,(\Delta_H+2\Lambda)=\delta^*\,(\Delta_V+2\Ein),
$$
and
$$
D\Ein(g)\, \delta^*=\frac12\delta^*\,(\Delta_H+2\Lambda)+\frac12\delta^*\,\Delta_V
=\delta^*\,(\Delta_V+\Ein)=\delta^*\,(\Delta+\Lambda)
$$
thus the ``stability'' of $\Im\delta^*$ by the two operators above.

When restricted on  the kernel of $\delta\,G$, we trivially have
$$
D\Ein(g)=\frac12(\Delta_L+2\Lambda),
$$
but also, by linearising  $\delta\, G\,\Ein(g)=0$  for instance, 
$$
\delta\,G\, (\Delta_L+2\Lambda)=0,
$$
then the stability of $\ker \delta\,G$.

We can also remark that with the formula above, 
 if $v\in V$, 
$$(\Delta_L+2\Lambda)(\delta^*v)=2\delta^*(\Ein v)=2\Ein\delta^*v$$ and $$D\Ein(g)(\delta^*v)=\delta^*(\Ein v)=\Ein\delta^*v.$$
For any function $u$, it is well known that  $\Delta_L(ug)=(\Delta u)g$, so
$$
(\Delta_L+2\Lambda)(d^*w \,g)=[(\Delta+2\Lambda) d^*w]g=[d^*(\Delta_H+2\Lambda)w]g.
$$
We obtain that 
$$
(\Delta_L+2\Lambda)G\delta^*=G\delta^*(\Delta_H+2\Lambda).
$$
If $v\in V$, we deduce
$$
(\Delta_L+2\Lambda)G\delta^*v=G\delta^*(2\Ein v).
$$
Assume that 
$-2\Lambda$ is not an eigenvalue of $\Delta_L$, 
then  $\Delta_L+2\Lambda$ is an isomorphism from $H^{k+2,p}(M,\mathcal S_2)$
to $H^{k,p}(M,\mathcal S_2)$. The image of the splitting  in Lemma \ref{lemmegrossplit} by $\Delta_L+2\Lambda$ produce:\footnote{Here also we have to replace the first factor by $(\Delta_L+2\Lambda)W$ when a choice of $W$ was made in the first factor of $K$.}
$$
H^{k,p}(M,\mathcal S_2)=\frac{(\ker \delta G \cap H^{k,p}(M,\mathcal S_2))}{\delta^*(\Ein V)}\oplus \delta^*(H^{k+1,p}(M,\mathcal T_1))\oplus G\delta^*(\Ein V).
$$
The two first factors are the same than  the image by $D\Ein(g)$ of the corresponding  spaces
in Lemma \ref{lemmegrossplit}. Let us study the image of  third one.
For $v\in V$, we compute
$$
\delta^*\delta G G\delta^* v=\delta^*\delta G (\delta^* v+\frac12d^*v\,g)
=\frac12\delta^*\delta G (d^*v\,g)=\frac{2-n}4\delta^*\delta (d^*v\,g)
$$
\begin{equation}\label{bianchisurV}
=\frac{n-2}4\delta^*dd^*v=\frac{n-2}2\delta^*\delta \delta^*v=-G\delta^*\delta \delta^*v.
\end{equation}
We deduce  for instance that
\begin{equation}\label{eqDEin}
D\Ein(g)G\delta^*V=\left[G\delta^*(\Ein\,.)+\frac {n-2}2\delta^*\delta\delta^*\right]V.
\end{equation}
Let us define 
$$
\mathcal F:=\delta^*(H^{k+1,p}(M,\mathcal T_1))\oplus G\delta^*(\Ein V).$$
We now prove that 
\begin{equation}\label{decompoF2}
\mathcal F=
\delta^*(H^{k+1,p}(M,\mathcal T_1))\oplus D\Ein(g)G\delta^*V.
\end{equation}
The fact that $\mathcal F$ is the sum of the two factors is clear by (\ref{eqDEin}).
Let $w$ in the intersection of the two factors, so 
$$
w=\delta^*u=G\delta^*\Ein v+\delta^*\delta\delta^*\frac{n-2}2v,
$$
for some $u\in H^{k+1,p}(M,\mathcal T_1)$ and $v\in V$.
Because of the decomposition of $\mathcal F$, we deduce
that $G\delta^*\Ein v=0$ thus $(\Delta_L+2\Lambda)G\delta^*v=0$, then 
$G\delta^*v=0$ and finally, by (\ref{bianchisurV}), $\frac{n-2}2\delta^*\delta\delta^*v=0$ so $w=0$.
We have obtained 
$$
H^{k,p}(M,\mathcal S_2)=\Im DF(0)\oplus \delta^*(H^{k+1,p}(M,\mathcal T_1)).
$$
We claim that $DF(0)$ is injective. Indeed, let  $h$ in the kernel of 
$DF(0)$, then $h=[u]+G\delta^*v$ with $[u]$ in the first summand of $K$. Thus $[\Delta_L+2\Lambda][u]+D\Ein(g) G\delta^*v=0$ so because of the decomposition (\ref{decompoF2}), we obtain $[\Delta_L+2\Lambda][u]=D\Ein(g) G\delta^*v=0$.
Its implies $[u]=0$ and  from equation (\ref{eqDEin}),
$v\in G\delta^*\Ein V\cap \delta^*(H^{k+1,p}(M,\mathcal T_1))=\{0\}$, so $h=0$.
\end{proof}
From the Lemma \ref{lemme23} with $E=\Ein({g})$ and Lemma \ref{imric} we directly deduce :
\begin{lemma}\label{lemmeproxdif}
If ${\mathcal E}\in H^{k,p}$ and $|{\mathcal E}-\Ein({g})|_{k,p}<\epsilon$, then there exist a metric 
$\griem\in H^{k+2,p}$ and a diffeomorphism $\varphi\in H^{k+1,p}$ for which $\Ein(\griem)=\varphi^*{\mathcal E}$.
\end{lemma}

We will complete the proof of proposition \ref{einsclose}, where now ${\mathcal E}\in H^{k+1,p}$,
but $\griem$ and $\varphi$ still cames from Lemma \ref{lemmeproxdif} so $\varphi$ is {\it a priori} not  regular enough.
If we  inspect the pages 364-365
in \cite{DeturckEinstein} we can see that we just have to change $\Ric(\griem)$ by $\Ein(\griem)$
to obtain that $\varphi$  is in fact in $H^{k+2,p}$. We conclude that $(\varphi^{-1})^*\griem\in H^{k+1,p}$
and has ${\mathcal E}$ as its image by $\Ein$. At this level we also use that $\Ein(\griem)$
is non degenerate (see equation (2.8) there).
\\

The Theorem \ref{maintheoremRicci} is now a direct consequence of the Lemma \ref{maintheoremPseudoRiem}, the Proposition \ref{einsclose} with $k=0$,
and the regularity result of \cite{Deturck-Kazdan}.

\begin{exemple}
Recalling that the Ricci curvature of a	 product of Riemannian manifolds is the direct sum of the
Ricci curvatures of each factors, we see that  a product of Einstein manifolds clearly satisfies the assumption of the Theorem \ref{maintheoremRicci}. 
The simplest example combining the 3 possibilities of Einstein constants is the following.
Let us consider three  compact Einstein manifolds $(\mathcal X,g_{-})$, $(\mathcal Y,g_{+})$, 
$(\mathcal Z,g_{0})$ with Ricci curvatures given by $\Ric(g_-)=-g_-$, $\Ric(g_+)=g_+$, $\Ric(g_0)=0$.
Then $M=\mathcal X\times\mathcal Y\times \mathcal Z$ endowed with 
$$
g=g_-\oplus g_+\oplus g_0,
$$
has  parallel Ricci curvature equal to 
$$
\Ric(g)=-g_-\oplus g_+\oplus 0.
$$
In this example, the kernel of $\Delta_L$ contains  the parallel tensors
$$
h=c_-g_-\oplus c_+g_+\oplus c_0g_0,
$$
for any constants $c_-,c_+,c_0$. Here, we only have to choose $\Lambda$ in order to destroy this kernel and make $\Ein(g)$ non degenerate.
\end{exemple}
%



\providecommand{\bysame}{\leavevmode\hbox to3em{\hrulefill}\thinspace}
\providecommand{\MR}{\relax\ifhmode\unskip\space\fi MR }
\providecommand{\MRhref}[2]{%
  \href{http://www.ams.org/mathscinet-getitem?mr=#1}{#2}
}
\providecommand{\href}[2]{#2}


\begin{thebibliography}{10}

\bibitem{Baldes1986}
Alfred Baldes, \emph{Nonexistence of {R}iemannian metrics with prescribed
  {R}icci tensor}, Nonlinear problems in geometry ({M}obile, {A}la., 1985),
  Contemp. Math., vol.~51, Amer. Math. Soc., Providence, RI, 1986, pp.~1--8.
  \MR{848927 (87k:53085)}

\bibitem{Delanoe1991}
Ph. Delano{\"e}, \emph{Obstruction to prescribed positive {R}icci curvature},
  Pacific J. Math. \textbf{148} (1991), no.~1, 11--15.

\bibitem{Delanoe2003}
\bysame, \emph{Local solvability of elliptic, and curvature, equations on
  compact manifolds}, J. Reine Angew. Math. \textbf{558} (2003), 23--45.
  \MR{1979181 (2004e:53054)}

\bibitem{Delay:ricciAE}
E.~Delay, \emph{Inversion d'op\'erateurs de courbure au voisinage de la
  m\'etrique euclidiennne}, bull. Soc. Math. France, \`a para\^{i}tre,
  hal-00973138.

\bibitem{Delay:ricciproduit}
\bysame, \emph{Sur l'inversion de l'op\'erateur de {R}icci au voisinage d'une
  m\'etrique {R}icci parall\`ele}, Annales de l'institut Fourier, \`a
  para\^{i}tre, hal-00974707v2.

\bibitem{Delay:etude}
\bysame, \emph{Etude locale d'op\'erateurs de courbure sur l'espace
  hyperbolique}, J. Math. Pures Appli. \textbf{78} (1999), 389--430.

\bibitem{Delay:study}
\bysame, \emph{Study of some curvature operators in the neighbourhood of an
  asymptotically hyperbolic {E}instein manifold}, Advances in Math.
  \textbf{168} (2002), 213--224.

\bibitem{DelayHerzlich}
E.~Delay and M.~Herzlich, \emph{Ricci curvature in the neighbourhood of
  rank-one symmetric spaces}, J. Geometric Analysis \textbf{11} (2001), no.~4,
  573--588.

\bibitem{Deturck:ricci}
D.~DeTurck, \emph{Existence of metrics with prescribed ricci curvature : Local
  theory}, Invent. Math. \textbf{65} (1981), 179--207.

\bibitem{Deturckrank1}
Dennis DeTurck and Hubert Goldschmidt, \emph{Metrics with prescribed {R}icci
  curvature of constant rank. {I}. {T}he integrable case}, Adv. Math.
  \textbf{145} (1999), no.~1, 1--97.

\bibitem{DeturckEinstein}
Dennis~M. DeTurck, \emph{Prescribing positive {R}icci curvature on compact
  manifolds}, Rend. Sem. Mat. Univ. Politec. Torino \textbf{43} (1985), no.~3,
  357--369 (1986).

\bibitem{Deturck-Kazdan}
Dennis~M. DeTurck and J.~Kazdan, \emph{Some regularity theorems in riemannian
  geometry}, Ann. Scient. Ec. Norm. Sup. \textbf{14} (1981), no.~4, 249--260.

\bibitem{Deturck-Koiso}
Dennis~M. DeTurck and Norihito Koiso, \emph{Uniqueness and nonexistence of
  metrics with prescribed {R}icci curvature}, Ann. Inst. H. Poincar\'e Anal.
  Non Lin\'eaire \textbf{1} (1984), no.~5, 351--359.

\bibitem{Hamilton1984}
Richard Hamilton, \emph{The {R}icci curvature equation}, Seminar on nonlinear
  partial differential equations ({B}erkeley, {C}alif., 1983), Math. Sci. Res.
  Inst. Publ., vol.~2, Springer, New York, 1984, pp.~47--72. \MR{765228
  (86b:53040)}

\bibitem{Lichnerowicz:prop}
A.~Lichnerowicz, \emph{Propagateurs et commutateurs en relativit\'e
  g\'en\'erale}, Pub. Math. de l'IHES \textbf{10} (1961), 5--56.

\bibitem{Pulemotov2013}
A.~Pulemotov, \emph{Metrics with prescribed {R}icci curvature near the boundary
  of a manifold}, Mathematische Annalen \textbf{357} (2013), 969--986.

\bibitem{PulemotovRubinstein}
A.~Pulemotov and Y.A. Rubinstein, \emph{Ricci iteration on homogeneous spaces},
  arXiv:1606.05064 [math.DG] (2016).

\bibitem{Wu:Holonomy}
H.~Wu, \emph{Holonomy groups of indefinite metrics}, Pacific J. Math.
  \textbf{20} (1967), 351--392.

\end{thebibliography}
\end{document}